\documentclass[11pt,a4paper]{llncs}
\usepackage[english]{babel}
\usepackage{verbatim}
\usepackage{psfrag}
\usepackage{graphicx}
\usepackage{xcolor}

\usepackage{amsmath}

\usepackage{amsfonts}
\usepackage{amssymb}
\def\F2{\mathbb{F}_{\hspace{-0.7mm}2}}
\newtheorem{algorithm}{Algorithm}

\def\Sym{\mbox{\rm Sym}}
\def\dim{\mbox{\rm dim}}

\def\FF{{\mathbb F}}

\def\f2{{\mathbb F}_{2}}

\newcommand{\AGL}{\mbox{\rm AGL}}

\newcommand{\GL}{\mbox{\rm GL}}

\newcommand{\ga}{\alpha}

\newcommand{\gd}{\delta}

\newcommand{\gl}{\lambda}
\newcommand{\gk}{\kappa}

\newcommand{\gs}{\sigma}

\newcommand{\gt}{\tau}

\newtheorem{fact}{Fact}

\begin{document}

\title{On differential uniformity of maps that may hide an algebraic trapdoor}


{\author{Marco Calderini \and Massimiliano Sala }} 


\institute{Department of Mathematics,
 University of Trento, Italy\\
  \email{marco.calderini@unitn.it, maxsalacodes@gmail.com }}

\authorrunning{M.~Calderini, M.~Sala}
\maketitle
\begin{abstract}
We investigate some differential properties for permutations in the affine group, of a vector space $V$ over the binary field, with respect to a new group operation $\circ$, inducing an alternative vector space structure on $V$.
\end{abstract}


\section{Introduction}
\label{intro}

Most modern block ciphers are built using components whose cryptographic strength is evaluated in terms of the resistance offered to attacks on the whole cipher. For example, differential properties of Boolean functions are studied for the S-Boxes to thwart differential cryptanalysis (\cite{des,inverse}).

Little is known on similar properties to avoid trapdoors in the design of the block cipher. In  \cite{CGC-cry-art-carantisalaImp} the authors investigate the minimal properties for the S-Boxes (and the mixing layer) of an AES-like cipher (more precisely, a translation-based cipher, or tb cipher) to thwart the trapdoor coming from the imprimitivity action, first noted in \cite{CGC-cry-art-paterson1} .

In \cite{li2003finite}, Li observed that if $V$ is a finite vector space over a finite field $\mathbb{F}_p$, the symmetric group $\Sym(V)$ will contain many isomorphic copies of the affine group $\AGL(V)$, which are its conjugates in $\Sym(V)$. So there are several structures $(V,\circ)$ of a $\mathbb{F}_p$-vector space on the set $V$ , where $(V,\circ)$ is the abelian additive group of the vector space. Each of these structure will yield in general a different copy $\AGL(V, \circ)$ of the affine group within $\Sym(V )$. So, a trapdoor coming from an alternative vector space structure, which we call \emph{hidden sum}, can be embedded in a cipher, whenever the permutation group generated by the round functions of the cipher is contained in a conjugate of $\AGL(V)$. In \cite{CGC-cry-art-carantisalaONan} the authors provide conditions on the S-Boxes of a tb cipher that avoid attacks coming from hidden sums. This result has been generalized to tb ciphers over any field in \cite{aragona2014group}. 
Also, in \cite{anti-crooked}, the authors studied such trapdoors, characterizing a new class of vectorial Boolean functions, which they call \emph{anti-crooked}, able to avoid any hidden sum.

In the yet unpublished Ph.D thesis \cite{phd} the author investigated some properties of affine groups, of a vector space over the binary field, with respect to a hidden sum $\circ$. In particular, he focused on affine groups which contain the translation group with respect to the usual sum $+$, and affine groups whom translation group is contained in $\AGL(V)$. In this paper we study the differential properties of maps which are affine w.r.t. a hidden sum. Our results are presented in Section 3, while in Section 2 we provide some preliminaries from previous works. Our main result, Theorem \ref{th:diff}, concludes Section 3. Section 4 concludes this paper with the sketch of an actual attack to a cipher in which a hidden sum trapdoor is embedded.

\section{Preliminaries}

Here we give some notation and some known results that we are going to use along the paper. In the following, if not specified, $V$ will be an $n$-dimensional vector space over $\F2$.

With the symbol $+$ we refer to the usual sum over the vector space $V$, and we denote by $T_+$, $\AGL(V,+)$ and $\GL(V,+)$, respectively, the translation, affine and linear groups w.r.t. $+$. 

We recall that a $p$-elementary group $G$ acting on a set $\Omega$ is a group of permutations on $\Omega$ such that for all $g$ in $G$ we have $g^p=Id_{\Omega}$.\\
A group $G$ is called regular if for all $a$ and $b$ in $\Omega$ there exists a unique $g$ in $G$ such that $g(a)=b$.
\begin{remark}
An elementary group acting on a vector space $V=\FF_p^n$ is obviously a $p$-elementary group. The translation group of V is an elementary abelian regular group. Vice versa, we claim that if $T$ is an elementary abelian regular group, there exists a vector space structure $(V,\circ)$ such that $T$ is the related translation group. In fact, from the regularity of $T$ we have $T=\{\gt_a\,|\,a\in V\}$ where $\gt_a$ is the unique map in $T$ such that $0\mapsto a$. Then, defining the sum $x\circ a:=\gt_a(x)$, it is easy to check that $(V,\circ)$ is a commutative group, and so we can consider the group operation as a sum, making it an
additive group without loss of generality. Moreover, let the multiplication of a vector by an element of $\FF_p$ defined by 
$$s v:=\underbrace{v\circ\dots\circ v}_{s},\mbox{  for all }s \in \FF_p,$$ then it is easy to check that for all $s,t\in \FF_p$, and $v,w\in V$
$$
s(v\circ w)=s v\circ s w, 
$$
$$
(s+t)v=s v\circ t v,
$$
$$
(st)v=s(t v)\,
$$
and being $T$ $p$-elementary $p v=0$. Thus $(V,\circ)$ is a vector space over $\FF_p$. Observe that $(V,\circ)$ and $(V,+)$ are isomorphic vector space (since $|V|<\infty$).
\end{remark}

For abelian regular subgroups of the affine group in \cite{CGC-alg-art-affreg} the authors give a description of these in terms of commutative associative algebras that one can impose on the vector space $(V,+)$ or, in other words, of products that can be defined on $V$ and distribute the sum $+$. We report the principal result shown in \cite{CGC-alg-art-affreg}. Recall that a (Jacobson) radical ring is a ring $(V,+,\cdot)$ in which every element is invertible with respect to the circle operation $x\circ y =x + y + x\cdot y$, so that $(V,\circ)$ is a group. The circle operation may induce a vector space structure on $V$ or not.

\begin{theorem}\label{th:somme}
Let $\mathbb{F}$ be an arbitrary field, and $(V,+)$ a vector space of arbitrary dimension over $\mathbb{F}$.

There is a one-to-one correspondence between
\begin{itemize}
\item[1)] abelian regular subgroups $T$ of $\AGL(V,+)$, and
\item[2)] commutative, associative $\mathbb{F}$-algebra structures $(V,+,\cdot)$ that one can impose on the vector space structure $(V,+)$, such that the resulting ring is radical.
\end{itemize}

In this correspondence, isomorphism classes of $\mathbb{F}$-algebras correspond to conjugacy classes under the action of $\GL(V,+)$ of abelian regular subgroups of $\AGL(V,+)$.
\end{theorem}

We recall that an exterior algebra over an $\FF$-vector space $V$ is the $\FF$-algebra whose product is the wedge product $\wedge$ having the following properties:
\begin{itemize}
\item[1)] $x\wedge x=0$ for all $x\in V$,
\item[2)]  $x\wedge y = - y\wedge x$.
\end{itemize}
The elements of the exterior algebra over $V$ are linear combinations of monomials such as $u, v \wedge w, x \wedge y \wedge z$, etc., where $u, v, w, x, y$, and $z$ are vectors of $V$. 

\begin{remark}\label{rm:algebra}
From the theorem above we can note that in characteristic $2$, algebras corresponding to elementary abelian regular subgroups of $\AGL(V,+)$ are exterior algebras or a quotient thereof.
\end{remark}

We will denote by $\gs_a$ the translation in $T_+$ such that $x\mapsto x+a$.  We will use $T_\circ$ and $\AGL(V,\circ)$ to denote the translation and affine group corresponding to a hidden sum $\circ$, that is when $(V,\circ)$ is a vector space and so $T_\circ$ is elementary abelian and regular. \\
As noted in the remark above, since $T_\circ$ is regular, for each $a\in V$ there is a unique map $\gt_a\in T_\circ$ such that $0\mapsto a$. Thus
$$
T_\circ=\{\gt_a\mid a\in V\}.
$$ 
The relation between $T_\circ$ and $\AGL(V,\circ)$ is that $\AGL(V,\circ)$ is the normalizer of $T_\circ$ in $\Sym(V)$, that is $\AGL(V,\circ)$ is the largest subgroup of $\Sym(V)$ contaning $T_\circ$ such that $T_\circ$ is normal in it. Indeed, $\AGL(V,+)$ is the normalizer of $T_+$ and they are, respectively, the isomorphic images of $\AGL(V,\circ)$ and $T_\circ$.
With $1_V$ we will denote the identity map of $V$.

\begin{remark}
If $T_\circ\subseteq \AGL(V,+)$, then $\gt_a=\gs_a\gk$ for some $\gk\in \GL(V,+)$, since $\AGL(V,+)=\GL(V,+)\ltimes T_+$. We will denote by $\gk_a$ the linear map $\gk$ corresponding to $\gt_a$.
\end{remark}

Let $T\subseteq \AGL(V,+)$ and define the set

$$
U(T)=\{a\mid \gt=\gs_a,\gt \in T\}.
$$

It is easy to check that $U(T)$ is a subspace of $V$, whenever $T$ is a subgroup. If $T=T_\circ$ for some operation $\circ$, then $U(T_\circ)$ is not empty for the following lemma.

\begin{lemma}[\cite{CGC-alg-art-affreg}]\label{lm:car}
Let $T_+$ be the group of translation in $\AGL(V,+)$ and let $T\subseteq \AGL(V,+)$ be a regular subgroup. Then, if $V$ is finite $T_+\cap T$ is nontrivial.
\end{lemma}

$U(T_\circ)$ is important in the context of our theory and its dimension gives fundamental information on the corresponding hidden sum.

\section{On the differential uniformity of a $\circ$-affine map}

Any round function of a translation-based block cipher (Definition $3.1$ \cite{CGC-cry-art-carantisalaImp}) is composed by a parallel s-Box $\gamma$, a mixing layer $\gl$ and a translation $\gs_k$ by the round key. The map $\gamma$ must be as non-linear as possible to create confusion in the message. An important notion of "non-linearity" of Boolean functions is the differential uniformity.

In this section we establish a lower bound on the differential uniformity of the maps lying in some $\AGL(V,\circ)$.
We will consider the two cases of affine group $\AGL(V,\circ)$ such that $T_\circ\subseteq \AGL(V,+)$ and/or $T_+\subseteq \AGL(V,\circ)$. In both cases in the following proofs we can consider w.l.o.g. maps $f$ such that $f(0)=0$. In fact in the first case we can compose $f$ with $\gt_{f(0)}$ that maps $f(0)$ to $0$ and in the second case we compose with $\gs_{f(0)}$, in both cases we compose with an affine map.

We recall the definition of differential uniformity.

\begin{definition}
Let $m,n\ge 1$. Let $f:\F2^m\to\mathbb{F}_2^n$, for any $a\in\F2^m$ and $b\in \F2^n$ we define
$$
\gd_f(a,b)=|\{x\in\F2^m \,|\,f(x+a)+f(x)=b\}|.
$$
The \emph{differential uniformity} of $f$ is 
$$
\gd(f)=\max_{\substack{a \in\F2^m,\, b\in \F2^n\\
a\neq 0}}\gd_f(a,b).
$$
$f$ is said $\gd$-\emph{differential uniform} if $\gd=\gd(f)$.
\end{definition}

We are ready for our first result.
\begin{lemma}\label{lm:diff2}
Let $T_\circ\subseteq \AGL(V,+)$ and $\dim(U(T_\circ))=k$. Then $f\in \AGL(V,\circ)$ is at least $2^k$ differentially uniform.
\end{lemma}
\begin{proof} 
Let $a\in U(T_\circ)$, then
$$
f(x+a)+f(x)=f(x\circ a)+f(x)=(f(x)\circ f(a))+f(x).
$$

\noindent So, for all $f(x) \in U(T_\circ)$ we have
$$
(f(x)\circ f(a))+f(x)=(f(x)+f(a))+f(x)=f(a),
$$
that implies $|\{x \,|\, f(x+a)+f(x)=f(a)\}|\ge 2^k$.
\end{proof}

When $T_+\subseteq \AGL(V,\circ)$, we can define $U_\circ(T_+)=\{a\mid\gs_a\in T_+\cap T_\circ\}$ and it is a vector subspace of $(V,\circ)$. Then we obtain, analogously, the following lemma.

\begin{lemma}\label{lm:diff3}
Let $T_+\subseteq \AGL(V,\circ)$ and $\dim(U_\circ(T_+))=k$, as a subspace of $(V,\circ)$. Then $f\in \AGL(V,\circ)$ is at least $2^k$ differentially uniform.
\end{lemma}

Recalling that given a ring $R$, $r\in R$ is called nilpotent if there exists an integer $n$ such that $r^n=0$, while $r\in R$ is called 
 unipotent if and only if $r-1$ is nilpotent, we have the following:

\begin{lemma}\label{lm:uni}
Let $T_\circ\subseteq \AGL(V,+)$. Then for each $a\in V$, $\gk_a$ has order $2$ and it is unipotent.
\end{lemma}
\begin{proof}
We know that $\gt_a$ has order $2$, because $T_\circ$ is elementary. Then, $\gt_a^2=1_V$ implies $\gt_a(a)=0$, and in particular $\gk_a(a)=a$. So 
$$
x= \gt_a^2(x)=\gk_a(\gk_a(x)+a)+a=\gk_a^2(x)+a+a=\gk_a^2(x)\quad \mbox{ for all }x \in V.
$$
That implies $(\gk_a-1_V)^2=\gk_a^2-1_V=0$.
\end{proof}

\begin{remark}
The lemma above can be easily generalized to any characteristic $p$, in this case the order of $\gk_a$ would be $p$.
\end{remark}

\begin{remark}
It is well known that a square matrix is unipotent if and only if its characteristic polynomial $P(t)$ is a power of $t -1$, i.e. it has a unique eigenvalue equals to $1$.
\end{remark}

We recall the following definition.

\begin{definition}
Let $A$ be an $n\times n$ matrix over a field $\mathbb{F}$, with $\gl\in \mathbb{F}$ along the main diagonal and 1 along the diagonal above it, that is
$$
A=\left[\begin{array}{ccccc}
						\gl & 1 &            &\dots & 0\\
						0 & \gl        &1  &\dots & 0\\
					\vdots&  & &   & \vdots \\
						
						0 & \dots &           &     &   \gl \end{array}\right].
$$
Then $A$ is called the $n\times n$ elementary Jordan matrix or Jordan block of size $n$.
\end{definition}

\begin{definition}
A matrix $A$ defined over a field $\mathbb{F}$ is said to be in Jordan canonical form if $A$ is block-diagonal where each block is a Jordan block defined over $\mathbb{F}$.
\end{definition}

The following theorem is well-known (see for instance \cite{lang}).
\begin{theorem}\label{th:jordan}
Let $A$ be an $n \times n $ matrix over a field $\mathbb{F}$ such that any eigenvalue of $A$ is contained in $\mathbb{F}$, then there exists a matrix $J$ defined over $\mathbb{F}$, which is in Jordan canonical form and similar to $A$.
\end{theorem}

\begin{lemma}\label{lm:fixpoint}
Let  $T_\circ\subseteq \AGL(V,+)$. Then for each $a \in V$, $\gk_a$ fixes at least $2^{\lfloor \frac{n-1}{2} \rfloor +1}$ elements of $V$.
\end{lemma}
\begin{proof}
From Lemma \ref{lm:uni}, $\gk_a$  has a unique eigenvalue equals to $1 \in \f2$, then from Theorem \ref{th:jordan} there exists a matrix over $\f2$ in the Jordan form similar to $\gk_a$. Thus, $\gk_a=AJA^{-1}$, for some $A,J\in GL(V,+)$ with

$$
J=\left[\begin{array}{ccccc}
						1 & \ga_1 &            &\dots & 0\\
						0 & 1        & \ga_2  &\dots & 0\\
					\vdots&  & &   & \vdots \\
						0 & \dots &           &  1 & \ga_{n-1}\\
						0 & \dots &           &     &   1\end{array}\right] \mbox{ and } J^2=\left[\begin{array}{cccccc}
						1 &  0 &     \ga_1 \ga_2      &   &\dots & 0\\
						0 & 1    &           0               & \ga_2\ga_3 & \dots & 0\\
					        \vdots&  & &   &   & \vdots \\
					        0 & \dots &         &  1 &  0 & \ga_{n-2}\ga_{n-1} \\
						0 & \dots &         &  &  1 & 0\\
						0 & \dots &       &    &     &   1\end{array}\right].
$$
where $\ga_i\in \mathbb{F}_2$ for $1\le i\le n-1$.

\noindent From the fact that $J$ is conjugated to $\gk_a$ we have $J^2=1_V$, and that implies $\ga_i\ga_{i+1}=0$ for all $1\le i\le n-2$.

\noindent Note that if $\ga_i=1 $ then $\ga_{i-1}$ and $\ga_{i+1}$ have to be equal to $0$. Thus we have that when $n$ is even at most $\frac{n}{2}$ $\ga_i$'s can be equal to $1$. Then at least $\frac{ n}{ 2}$ elements of the canonical basis are fixed by $J$. When $n$ is odd we have at most $\frac{n-1}{2}$ $\ga_i$'s equal to $1$ and then at least $\frac{n-1}{2}+1$ elements of the canonical basis are fixed by $J$.
Our claim follows from the fact that $\gk_a$ is conjugated to $J$.
\end{proof}

In terms of algebras we have the following corollary.

\begin{corollary}\label{lm:rewrited}
Let $T_\circ \subseteq \AGL(V,+)$, and let $(V,+,\cdot)$ be the associated algebra of Theorem \ref{th:somme}. Then for each $a \in V$, $a\cdot x$ is equal to $0$ for at least $2^{\lfloor \frac{n-1}{2}\rfloor +1}$ elements $x$ of $ V$.
\end{corollary}

\begin{remark}\label{rk:dim7}
The bound on the number of elements fixed by $\gk_a$ given in Lemma \ref{lm:fixpoint} is tight. In fact let $(V,+,\cdot)$ be the exterior algebra over a vector space of dimension three, spanned by $e_1,e_2,e_3$. That is, $V$ has basis
$$
e_1,e_2,e_3,e_1\wedge e_2, e_1\wedge e_3, e_2\wedge e_3, e_1\wedge e_2 \wedge e_3.
$$
We have that $e_1\cdot x =0$ for all $x \in E=\langle e_1,e_1\wedge e_2, e_1\wedge e_3,  e_1\wedge e_2 \wedge e_3 \rangle$. So, for all $x \in E$

$$
x\circ e_1=x+e_1+x\cdot e_1=x+e_1.
$$
Vice versa if $x\circ e_1=x+e_1$ then $x \in E$. The size of $E$ is $2^4$. 
\end{remark}

\begin{lemma}\label{lm:diff6}
Let $T_\circ\subseteq \AGL(V,+)$. Then $f\in \AGL(V,\circ)$ is at least $2^{\lfloor \frac{n-1}{2}\rfloor +1}$ differentially uniform.
\end{lemma}
\begin{proof}
From Lemma \ref{lm:car} there exists $a \in U(T_\circ)$ different from zero. So
$$
f(x+a)+f(x)=f(x\circ a)+f(x)=(f(x)\circ f(a)) + f(x)=
$$
$$
(f(x)+f(a)+f(a)\cdot f(x))+f(x)
$$

\noindent Now, from Corollary \ref{lm:rewrited} we have that $f(a)\cdot f(x) =0$ for at least $2^{\lfloor \frac{n-1}{2}\rfloor +1}$ elements of $V$.

\noindent This implies $|\{x \,|\, f(x+a)+f(x)=f(a)\}|\ge 2^{\lfloor \frac{n-1}{2}\rfloor +1}$.

\end{proof}

\begin{lemma}\label{lm:diff7}
Let $T_+\subseteq \AGL(V,\circ)$. Then $f\in \AGL(V,\circ)$ is at least $2^{\lfloor \frac{n-1}{2}\rfloor +1}$ differentially uniform.
\end{lemma}
\begin{proof}
Note that Theorem \ref{th:somme}, Lemma \ref{lm:car} and Corollary \ref{lm:rewrited} hold also inverting the operation $\circ$ and $+$. Then, there exists $a \in V$ different from zero such that $x+a=x\circ a $ for all $x\in V$. Considering the algebra $(V,\circ,\cdot)$ such that $x+y=x\circ y\circ x\cdot y$ for all $x,y\in V$, we have
$$
f(x+a)+f(x)=f(x\circ a)+f(x)=(f(x)\circ f(a)) + f(x)=
$$
$$
(f(x)\circ f(a)) \circ f(x)\circ f(x)\cdot (f(x) \circ f(a))=
$$
$$
f(x)\circ f(a) \circ f(x)\circ f(x)\cdot f(x) \circ f(x)\cdot f(a).
$$

\noindent From Remark \ref{rm:algebra}, we have $y^2=0$ for all $y \in V$, and from Corollary \ref{lm:rewrited} $f(x)\cdot f(a)=0$ for at least $2^{\lfloor \frac{n-1}{2}\rfloor +1}$ elements. Thus
$$|\{x \,|\, f(x+a)+f(x)=f(a)\}|\ge 2^{\lfloor \frac{n-1}{2}\rfloor +1}.$$
\end{proof}

Summarizing our results in this section, especially Lemma \ref{lm:diff2}, \ref{lm:diff3}, \ref{lm:diff6}, \ref{lm:diff7}, we obtain
our theorem on the claimed differentiability.

\begin{theorem}\label{th:diff}
Let $T_\circ\subseteq \AGL(V,+)$ ($T_+\subseteq \AGL(V,\circ)$, respectively). Let $f\in \AGL(V,\circ)$. Then $\gd(f)\ge 2^m$, where \\
\begin{itemize}
\item $m=\max\{{\lfloor \frac{n-1}{2}\rfloor +1},\dim (U(T_\circ))\}$
\item  ($m=\max\{{\lfloor \frac{n-1}{2}\rfloor +1},\dim (U_\circ(T_+))\}$, respectively).
\end{itemize}
\end{theorem}

By a computer check we obtain the following fact.
\begin{fact}\label{fact}
Let $V=\F2^n$ with $n=3,4,5$. If $T_+\subseteq \AGL(V,\circ)$, let $f\in \AGL(V,\circ)$. Then $\gd(f)\ge 2^{n-1}$.
\end{fact}

\begin{remark}
For $n=7,8$ there exist examples of functions that are affine w.r.t. a hidden sum $\circ$ satisfying $T_+\subseteq \AGL(V,\circ)$ and $\gd(f)=2^{n-2}$. The existence of these permutations and Fact \ref{fact} suggest  that
probably there may exist bounds which are sharper than those in Theorem \ref{th:diff}.
\end{remark}

\begin{remark}
Note that if we consider $f\in \Sym(\F2^4)$ with $\gd(f)=4$ then the parallel map $(f,f)$ acting on $\F2^8$ is $2^6$ differentially uniform. Thus the differential uniformity may not guarantee, alone, security from a hidden sum trapdoor!
\end{remark}

\section{A block cipher with a hidden sum}
\label{sec:4}

In this section we give an example, similar to that described in \cite{anti-crooked}, of a translation based block cipher in a small dimension, in which it is possible to embed a hidden-sum trapdoor.

Let $m=3$, $n=2$, then $d=6$ and we have the message space $V=\mathbb{F}_2^6$.
The mixing layer of our toy cipher is given by the matrix
$$
\gl=\left[\begin{array}{cccccc}
0& 0 &1 &1 &1 &0\\
0& 0& 1& 0& 1& 1\\
0& 0& 0& 0& 0& 1\\
1&0& 1& 0& 0& 1\\
1& 1& 1& 0& 0& 1\\
0& 0& 1& 0& 0& 0\end{array}\right]
$$

Note that $\gl$ is a proper mixing layer (see Definition $3.2$ \cite{aragona2014group}).
The bricklayer transformation $\gamma=(\gamma_1,\gamma_2)$ of our toy cipher is given by two identical S-boxes
$$
\gamma_1=\gamma_2=\ga^4x^6 + \ga^3 x^4 + \ga x^3 + \ga^3x^2 +  x + \ga^6
$$
where $\ga$ is a primitive element of $\mathbb{F}_{2^3}$ such that $\ga^3=\ga+1$.\\
The S-box $\gamma_1$ is $4$-differential uniform.

%
%
%
%
%
%
%
%
%
%
%
Consider the hidden sum $\circ$ over $V_1=V_2=(\mathbb{F}_2)^3$ induced by the elementary abelian regular group $T_\circ=\langle \tau_1,\tau_2,\tau_3\rangle$, where
\begin{equation}\label{eq:generatori}
\tau_1(x)=x\cdot\left[\begin{array}{ccc}
1& 0& 0\\
0 &1 &1\\
0 &0& 1\end{array}\right]+e_1,\;
\tau_2(x)=x\cdot\left[\begin{array}{ccc}
1& 0& 1\\
0& 1 &0\\
0& 0& 1\end{array}\right]+e_2,\;
\tau_3(x)=x\cdot\left[\begin{array}{ccc}
1& 0& 0\\
0 &1 &0\\
0 &0& 1\end{array}\right]+e_3,
\end{equation}
with $e_1=(1,0,0)$, $e_2=(0,1,0)$ and $e_3=(0,0,1)$. In other words, $\tau_i(x)=x\circ e_i$ for any $1\leq i\leq 3$.\\
Obviously $T=T_\circ\times T_\circ$ is an elementary abelian group inducing the hidden sum $(x_1,x_2)\circ'(y_1,y_2)=(x_1\circ y_1,x_2\circ y_2)$ on $V=V_1\times V_2$. By a computer check it results $\langle T_+, \gl \gamma\rangle \subseteq \AGL(V,\circ')$,
and $\circ'$ is a hidden sum for our toy cipher. It remains to verify whether it is possible to use it to attack the toy cipher with an attack that costs less than brute force. 
 We are considering a cipher where the number of rounds is so large to make any classical attack useless (such as differential cryptanalysis) and the key scheduling offer no weakness.  Therefore, the hidden sum will  actually be essential to break the cipher only if the attack that we build will cost significantly less than $64$ encryptions, considering that the key space is $\mathbb{F}_2^6$.
\begin{remark}
$T_\circ$ is generated by the translations corresponding to $e_1,e_2$ and $e_3$, which implies that the vectors $e_1,e_2,e_3$ form a basis for $(V_1,\circ)$. Let $x=(x_1,x_2,x_3)\in V_1$, from \eqref{eq:generatori} we can simply write
 {\small
$$
 \tau_1(x)=(x_1+1,x_2,x_2+x_3),\tau_2(x)=(x_1,x_2+1,x_1+x_3),\tau_3(x)=(x_1,x_2,x_3+1).
 $$}
 \noindent Let us write $x$ as a linear combination of $e_1$, $e_2$ and $e_3$ w.r.t. to the sum $\circ$, i.e. $x=\gl_1 e_1\circ \gl_2 e_2 \circ \gl_3 e_3$. We have that $\lambda_1=x_1$, $\lambda_2=x_2$ and $\lambda_3=\lambda_1\lambda_2+x_3$. So
 \begin{equation}\label{eq:comb}
 (x_1,x_2,x_3)=x=(\gl_1,\gl_2,\gl_1\gl_2+\gl_3)
 \end{equation}
\end{remark}

Thanks to the previous remark we can find the coefficients of  a vector $v'=(v,u)\in V$ with respect to $\circ'$ by using the following algorithm separately on the two bricks of $v'$.
\begin{algorithm}\label{alg:1}
\ \\
{\bf Input:} vector $x\in \mathbb{F}_2^3$\\
{\bf Output:} coefficients $\gl_1$, $\gl_2$ and $\gl_3$.\\
$[1]$ $\lambda_1\leftarrow x_1$;\\
$[2]$ $\lambda_2\leftarrow x_2$;\\
$[3]$ $\gl_3\leftarrow\gl_1\gl_2+x_3$;\\
return $\gl_1,\gl_2,\gl_3$.
\end{algorithm}

Let $v'=(v,u)\in V$, we write 
$$v=\gl_1^v  e_1\circ \gl_2^v e_2 \circ \gl_3^v e_3\mbox{ and }u=\gl_1^u  e_1\circ \gl_2^u e_2 \circ \gl_3^u e_3.
$$ 
We denote by 
$$
[v']=[\gl_1^{v},\gl_2^{v},\gl_3^{v},\gl_1^{u},\gl_2^{u},\gl_3^{u}]
$$
 the vector with the coefficients obtained from the bricks of $v'$ using Algorithm \ref{alg:1}. 
 
Let $\varphi=\varphi_k$ be the encryption function, with a given unknown session key $k$. We want to mount two attacks by computing the matrix $M$ and the translation vector $t$ defining $\varphi\in\AGL(V,\circ')$, so $t=\varphi(0)$ and  $[\varphi(x)]=[x]\cdot M+[t]$.\\
Assume we can call the encryption oracle. Then $M$ can be computed from the $7$ ciphertexts $\varphi(0),\varphi(e_1'),\dots,\varphi(e_6')$ (where $e_1'=(1,0,0,0,0,0),\dots,e_6'=(0,0,0,0,0,1)$), since the $([\varphi(e_i')]+[t])$'s represent the matrix rows. In other words, we will have 
 $$
 [\varphi(v')]=[v']\cdot M+[t],\quad [\varphi^{-1}(v')]=([v']+[t])\cdot M^{-1},
 $$
for all $v'\in V$, where the product row by column is the standard scalar product. The knowledge of  $M$, $t$ and $M^{-1}$ provides a global deduction (reconstruction), since it becomes trivial to encrypt and decrypt. In fact, to encrypt $v$ it is enough to compute $[v]$, applying $[v]\mapsto [v]\cdot M+[t]=[w]$ and then pass from $[w]$ to the standard representation $w$ via \eqref{eq:comb}. Analogously to decrypt. However, following \cite{anti-crooked}, we have an alternative depending on how we
compute $M^{-1}$, resulting in one attack with $7$ encryptions
and another with $7$ encryptions and $7$ decryptions. Both are much faster
than brute-force searching in the keyspace.


\end{document}